\documentclass[11pt]{article}

\usepackage{amsmath}
\usepackage{amssymb}
\usepackage{amsthm}
\usepackage[nottoc]{tocbibind}

\usepackage{enumerate}
\usepackage{bbm}
\usepackage{fancyhdr}

\usepackage{graphicx}
\graphicspath{ {../res/} }

\newtheorem{theorem}{Theorem}
\newtheorem*{thm*}{Theorem}
\newtheorem{lemma}[theorem]{Lemma}
\newtheorem{corollary}[theorem]{Corollary}
\newtheorem{proposition}[theorem]{Proposition}
\newtheorem{claim}{Claim}
\theoremstyle{definition}

\newtheorem*{claim*}{Claim}

\theoremstyle{plain}
\newtheoremstyle{named}{}{}{\itshape}{}{\bfseries}{.}{.5em}{\thmnote{#3}#1}
\theoremstyle{named}

\newcommand{\norm}[1]{\left\lVert#1\right\rVert}
\newcommand{\Z}{\mathbb Z}
\newcommand{\inprod}[1]{\left\langle#1\right\rangle}
\newcommand{\bad}{\mathsf{bad}}
\def\lin{\operatorname{lin}}
\newcommand{\E}{\mathbb{E}}
\newcommand{\R}{\mathbb{R}}
\newcommand{\N}{\mathbb{N}}
\newcommand{\PP}{\mathbb P}

\fancypagestyle{update}{\fancyhf{}\fancyfoot[R]{Last Updated: \today}}

\begin{document}

\title{Lyapunov exponents of orthogonal-plus-normal cocycles}
\author{Sam Bednarski and Anthony Quas}
\maketitle
\begin{abstract}
    We consider products of matrices of the form $A_n=O_n+\epsilon N_n$ where 
    $O_n$ is a sequence of $d\times d$ orthogonal matrices and $N_n$ has 
    independent standard normal entries and the $(N_n)$ are mutually independent. 
    We study the Lyapunov exponents of the cocycle as a function of $\epsilon$, giving
    an exact expression for the $j$th Lyapunov exponent in terms of the Gram-Schmidt
    orthogonalization of $I+\epsilon N$. Further, we study the asymptotics of these 
    exponents, showing that $\lambda_j=(d-2j)\epsilon^2/2+O(\epsilon^4|\log\epsilon|^4)$.
\end{abstract}

\section{Introduction and Statement of Results}

Lyapunov exponents play a highly important role in dynamical systems, allowing 
quantification of chaos, the development of a theory of hyperbolic and non-uniformly
hyperbolic dynamical systems and much more. We work in the framework of multiplicative
ergodic theory, where one has a base dynamical system $\sigma\colon\Omega\to\Omega$ 
preserving an ergodic measure $\mathbb P$ and a measurable map $A\colon \Omega\to M_{d\times d}(\R)$.
One then takes the cocycle of partial products $A^{(n)}_\omega$ of the sequence of $d\times d$ matrices and one studies
the limiting growth rate of the $j$th singular value of the products.

In the case $d=1$ this is often straightforward to calculate: the Lyapunov exponent is just
$\int\log|A(\omega)|\,d\PP(\omega)$. In higher dimensions, Lyapunov exponents tend to be much
harder to calculate, and it is rare to be able to give exact expressions. 

In this paper, we are able to establish exact expressions for Lyapunov exponents for
cocycles of a particular form, namely where the matrices $A_\omega$ are of the form 
$O_\omega+\epsilon N_\omega$, where the $O_\omega$ are orthogonal matrices
and the $N_\omega$ are mutually independent Gaussian matrices with independent standard
normal entries. We further assume that the $(N_{\sigma^n\omega})$ are independent of the 
cocycle $O_\omega$. We then interpret the cocycle as an additive noise perturbation of
a cocycle of orthogonal matrices. 

Our main results are the following:
\begin{theorem}\label{thm:expression}
Let $\sigma\colon\Omega\to\Omega$ be an ergodic transformation preserving an invariant
measure $\PP$ and let $O\colon\Omega\to O(d,\R)$ be a measurable map into the $d\times d$
orthogonal matrices. Suppose that $N\colon\Omega\to M_{d\times d}(\R)$ is measurable, and that
that conditioned on $(O_{\sigma^n\omega})_{n\in\Z}$ and $(N_{\sigma^n\omega})_{n\ne 0}$,
$N_\omega$ has independent standard normal entries. 
Then for all $\epsilon\in\R$, the Lyapunov exponents of the cocycle 
$A_\omega=O_\omega+\epsilon N_\omega$ are given by
$$
\lambda_j=\E\log\|c_j^\perp(I+\epsilon N)\|,
$$
where $c_j^\perp(A)$ is the $j$th column of the Gram-Schmidt orthogonalization of $A$. 
\end{theorem}

The following theorem describes the asymptotic behaviour of the exponents as $\epsilon$ 
tends to 0.

\begin{theorem}\label{thm:approx}
    Let the matrix cocycle be as above. 
    Then the Lyapunov exponents satisfy
    $$
    \lambda_j(\epsilon)=(d-2j)\tfrac {\epsilon^2}2+O(\epsilon^4|\log\epsilon|^4)
    $$
    as $\epsilon\to 0$.
\end{theorem}

We make the following conjecture. 
Let $\sigma$ be an ergodic measure-preserving transformation of
a space $(\Omega,\PP)$. If $B\colon\Omega\to M_{d\times d}(\R)$ is the generator
of a matrix cocycle with the property that $\|B_\omega\|\le 1$ almost surely,
and $N_\omega$ is Gaussian with the independence properties above, 
then setting $\lambda_j'(\epsilon)$ to be the sequence of Lyapunov exponents of
the cocycle $A^\epsilon_\omega=B_\omega+\epsilon N_\omega$, one has
$$
\lambda_j'(\epsilon)-\lambda_{j+1}'(\epsilon)\ge \lambda_j(\epsilon)-\lambda_{j+1}(\epsilon)
\text{ for all $\epsilon>0$,}
$$
where $\lambda_j(\epsilon)$ are the Lyapunov exponents for the cocycle 
described in Theorem \ref{thm:expression}.

That is, we conjecture that there are universal lower bounds on the
gaps between consecutive Lyapunov exponents of Gaussian perturbed cocycles
of matrices where the matrices in the unperturbed cocycle have norm at most 1; and 
that these lower bounds are obtained in the case where all of the matrices are
the identity matrix. 

The results in this paper are closely related to results in Newman \cite{Newman}, where 
he gave a result similar to Theorem \ref{thm:expression} for some i.i.d.\ cocycles
involving Gaussian matrices and SDE flows on the space of non-singular matrices. 
Newman also re-derives an important result of Dynkin \cite{Dynkin} 
that also has intermediate proofs due to
LeJan \cite{LeJan}; Baxendale and Harris \cite{BaxendaleHarris}; and Norris, Rogers and Williams \cite{NorrisRogersWilliams}.
Dynkin's result concerns the Lyapunov exponents of a simple stochastic flow on $GL_d(\R)$, the group of invertible $d\times d$ matrices,
and identifies explicit exact Lyapunov exponents for the flow. Although this cocycle is not the same as ours,
it is in the same spirit. The Lyapunov exponents in that paper
have a similar form to ours and are given by $\lambda_k=(d-2k+1)\sigma^2/2$
 
	\section{Definitions and Preliminary lemmas}
        If $N$ is a $d\times d$ matrix valued random variable whose entries are independent
        standard normal random variables, we will say that $N$ is the \emph{standard Gaussian matrix random variable}.
        We will need the following property of the normal distribution:
    \begin{lemma}\label{lem:samedist}
        Let $U$ be an orthogonal matrix and let $N$ be a standard Gaussian matrix random variable of the same dimensions.
        Then the matrices $N$, $UN$ and $NU$ are equal in distribution. 
    \end{lemma}
    This follows from a more general fact about the multivariate normal distribution.
    \begin{proposition}\label{prop:MVN}
        Let $X\sim N(\boldsymbol{\mu},\boldsymbol{\Sigma})$ be a $d$-dimensional multivariate normal distribution with mean vector $\boldsymbol{\mu}$ and covariance matrix $\boldsymbol{\Sigma}$. Suppose $V$ is a $d\times n$ matrix of rank $d$. Then $VX \sim N(V\boldsymbol{\mu},V\boldsymbol{\Sigma}V^T)$.
    \end{proposition}
    \begin{proof}
    Recall that $X\sim N(\boldsymbol{\mu},\boldsymbol{\Sigma})$ if and only if $X\sim AZ +\boldsymbol{\mu}$ where $AA^T = \boldsymbol{\Sigma}$ and $Z \sim N(\boldsymbol{0},I_d)$. If $X = AZ+\boldsymbol{\mu}$
    this implies that 
    \begin{align*}
        VX &= VAZ + V\boldsymbol{\mu} \\
        &\sim N(V\boldsymbol{\mu},VA(VA)^T)  \text{ by the fact above}\\
        &\sim N(V\boldsymbol{\mu},VAA^TV^T) \\
        &\sim N(V\boldsymbol{\mu},V\boldsymbol{\Sigma}V^T)
    \end{align*}
    \end{proof}

\begin{lemma}\label{lem:logNest}
  Let $N$ be a standard normal random variable. Then for any $a$ and for $b\ne 0$,
  $\E\log^-|a+bN|\le \E\log^-|bN|<\infty$.
\end{lemma}

\begin{proof}
    We have
    \begin{align*}
        \E\log^-|a+bN|&=\int_0^\infty \PP(\log^-|a+bN|<t)\,dt\\
        &=\int_0^\infty\PP(N\in[-a-e^{-t}/|b|,-a+e^{-t}/|b|]\,dt\\
        &\le\int_0^\infty \PP(N\in [-e^{-t}/|b|,e^{-t}/|b|]\,dt\\
        &=\E\log^-|bN|.
    \end{align*}
    Since $\log^-|bN|\le\log^-|b|+\log^-|N|$, and
    $\E\log^-|N|=\frac{2}{\sqrt{2\pi}}\int_0^1 |\log x| e^{-x^2/2}\,dx$, 
    it is easy to see $\E\log^-|bN|<\infty$.
\end{proof}

For a matrix $B$, let $c_j(B)$ denote the $j$th column of $B$ and
let $\theta_j(B)=\text{dist}\big(c_j(B),\lin(\{c_i(B)\colon i\ne j\})\big)$
and let $\Theta(B)=\min\theta_j(B)$. 

\begin{lemma}\label{lem:Thetaest}
Let $A$ be an arbitrary matrix and let $\epsilon>0$. Let $Z$ denote a
standard $d\times d$ Gaussian matrix random variable and $N$ denote a standard normal
random variable. 
Then $\E\log^-\Theta(A+\epsilon Z)\le \E\log^-(\epsilon N)<\infty$.

Further, if $S$ is any set, $\E\big(\log^-\Theta(A+\epsilon Z)\,\mathbf 1_S\big)
\le d\,\PP(S)(1+\log^-(\epsilon\PP(S)))$.
\end{lemma}

\begin{proof}
    Let $\mathcal F$ denote the $\sigma$-algebra generated by the columns of $N$ except 
    for the $j$th. Then $\E\log^-\theta_j(A+\epsilon Z)=
    \E\big(\E(\log^-\theta_j(A+\epsilon Z)|\mathcal F)\big)$.
    Let $\mathbf n$ be an $\mathcal F$-measurable unit normal to the subspace spanned
    by $(c_i(A+\epsilon Z))_{i\ne j}$ (this is almost surely unique up to a change of sign). 
    Then $\theta_j(A+\epsilon Z)=|\langle \mathbf n,c_j(A+\epsilon Z)\rangle|
    =|\langle \mathbf n,c_j(A))\rangle+\epsilon\langle \mathbf n,c_j(Z)\rangle|$.
    Let $a=\langle \mathbf n,c_j(A))\rangle$ (an $\mathcal F$-measurable random variable)
    and note that since $c_j(Z)$ is independent of the unit vector $\mathbf n$,
    conditioned on $\mathcal F$, by Proposition \ref{prop:MVN},
    $\langle \mathbf n,c_j(Z)\rangle$ is distributed
    as a standard normal random variable. Hence we have
    $\E\log^-\theta_j(A+\epsilon Z)=\E\big(\E(\log^-\theta_j(A+\epsilon Z)\big|\mathcal F)\big)
    =\E\big(\E(\log^-|a+\epsilon N|\big|\mathcal F)\big)\le\E\big(\E(\log^-|\epsilon N|\big|\mathcal F)\big)
    =\E\log^-|\epsilon N|$ which is finite by the lemma above.

    By definition, $\Theta(A+\epsilon Z)=\min_j\theta_j(A+\epsilon Z)$ so that
    $\log^-\Theta(A+\epsilon Z)=\max_j\log^-\theta_j(A+\epsilon Z)\le\sum_j\log^-\theta_j(A+\epsilon Z)$.
    By Lemma \ref{lem:logNest}, we see $\E\log^-\Theta(A+\epsilon Z)<\infty$ as required.

    Now if $S$ is any set, we have
    \begin{align*}
        \E\big(\log^-\theta_j(A+\epsilon Z)\,\mathbf 1_S\big)&=
        \int_0^\infty \PP(S\cap\{\log^-\theta_j(A+\epsilon Z)>t\})\,dt\\
        &\le \int_0^\infty \min\big(\PP(S),\PP(\theta_j(A+\epsilon Z)<e^{-t})\big)\,dt\\
        &\le \int_0^\infty \min\big(\PP(S),\PP(a+\epsilon N\in [-e^{-t},e^{-t}])\big)\,dt\\
        &\le \int_0^\infty \min\big(\PP(S),e^{-t}/\epsilon\big)\,dt,
    \end{align*}
    where in the third line, as above, $a$ is a random variable that is independent of $N$. For the fourth line,
    we used the fact that the density of a standard normal is bounded above by $(2\pi)^{-1/2}<\frac 12$.
    Separating the integration region into $[0,\log^-(\epsilon \PP(S))]$ and $[\log^-(\epsilon\PP(S)),\infty)$,
    we obtain $\E\big(\log^-\theta_j(A+\epsilon Z)\,\mathbf 1_S\big)\le \PP(S)\log^-(\epsilon\PP(S))+\PP(S)$.
    Since $\log^-\Theta(B)\le\sum_{j=1}^d\log^-\theta_j(B)$, the result follows. 
\end{proof}

For any vector $y$ in $\R^d$, $y$ has at
least one coefficient of magnitude $\|y\|/\sqrt d$, say the $j$th, so 
$\|By\|\ge \|y_jc_j(B)+\sum_{i\ne j}y_ic_i(B)\|\ge |y_j|\theta_j(B)\ge
(1/\sqrt d)\Theta(B)\|y\|$. If $B$ is invertible then $\Theta(B)$ is 
non-zero and substituting $y=B^{-1}x$ gives $\|B^{-1}\|\le \sqrt d/\Theta(B)$.

    \begin{corollary}\label{cor:Margulis-suff}
        Let $(A_n)$ denote an i.i.d. sequence of $d\times d$ random matrices
        where $A_n=I+\epsilon N_n$, and where $N_n$ is a $d\times d$ standard Gaussian matrix random 
        variables. Then $(A_n)$
        satisfies the following:
        \begin{enumerate}
            \item $A_n\in GL_d(\R)$ a.s.;
            \item the distribution of $A_n$ is fully supported in $GL_d(\R)$: for any non-empty open set $U\subset GL_d(\R)$, 
            $\PP(A_n\in U)>0$. 
            \item $\log\|A_n\|\in L^1(\Omega)$.
        \end{enumerate}
    \end{corollary}

This corollary establishes that the sequence $(A_n)$ satisfies the conditions of the Gol'dsheid-Margulis theorem \cite[Theorem 5.4]{GoldsheidMargulis}
which ensures that the Lyapunov exponents of the cocycle $A^{(n)}=(I+\epsilon N_n)\cdots (I+\epsilon N_1)$
are all distinct.
    
    \begin{proof}
        The distribution of the matrices $A_i$ is mutually absolutely continuous with respect
        to Lebesgue measure. Since the zero locus of the polynomial equation $\det(A)=0$ is a measure zero set,
        the first and second conclusions are established. 
        To show $\log\|A_i\|$ is integrable, we separately show that $\log^+\|A_i\|$ and $\log^-\|A_i\|$
        are integrable. First,
        $$
        \E \log^+\norm{A_1}
        \leq \E \norm{A_1}
        \leq \E \sum_{1\leq i,j\leq d}|A_{ij}|
        $$
        where each $A_{ij}$ is an integrable normal random variable. 
        The fact that $\E\log^+\norm{A_1^{-1}}<\infty$ follows from Lemma \ref{lem:Thetaest} and the 
        observation that $\|A_1^{-1}\|\le \sqrt d/\Theta(A_1)$ made above.
    \end{proof}

We make extensive use of the singular value decomposition in what follows. 
More information on this topic may be found in Horn and Johnson \cite{horn_johnson_1991}
and Bhatia \cite{Bhatia}.
For a $d\times d$ matrix $A$, a singular value decomposition is a triple $(L,D,R)$
where $L$ and $R$ are orthogonal matrices and $D$ is a diagonal matrix 
with non-negative entries such that
$A=LDR$. We impose without loss of generality 
the requirement that the diagonal entries of $D$ are decreasing. The matrix $D$ is 
uniquely determined by $A$ while there is some freedom in the choice of $L$ and $R$.
The \emph{singular values} of $A$ are denoted $s_i(A)$, where $s_i(A)$ is the
$i$th entry of the diagonal of $A$. It is known (see for example
Ragunathan \cite[Lemma 1]{Raghunathan}) that there exist measurable functions
$L$, $D$ and $R$ mapping $M_d(\R)$ to $O(d,\R)$, $M_\text{diag}(d,\R)$
and $O(d,\R)$ respectively such that $A=L(A)D(A)R(A)$.

It is well known that $|s_i(A)-s_i(B)|\le \|A-B\|$ where
$\|\cdot\|$ is the standard operator norm on matrices. 
We also make use of the products $S_i^j(A)=s_i(A)\cdots s_j(A)$. 
These have an interpretation in terms of exterior algebra. 
We write $\bigwedge^k\R^d$ for the $k$th exterior power of $\R^d$
and equip it with the standard inner product coming from the Cauchy-Binet formula
and the corresponding norm. In particular if $v_1,\ldots,v_d$ is an orthonormal 
basis for $\R^d$, then $\{v_{i_1}\wedge\cdots\wedge v_{i_k}\colon
i_1<i_2<\ldots<i_k\}$ is an orthonormal basis for $\bigwedge^k\R^d$. 
With respect to the corresponding operator norm, it is well known that
$\big\|A^{\wedge k}\|=S_1^k(A)$.

If $(A_n)$ is an independent identically distributed sequence of random variables
taking values in $GL_d(\R)$ such that $\E\log\|A_1\|^{\pm 1}<\infty$,
it was shown by Oseledets \cite{Oseledets} and Raghunathan \cite{Raghunathan}
that the limits 
$$
\lim_{n\to\infty}\frac 1n\log s_j(A_n\ldots A_1)
$$
exist and are almost surely constant for almost every realization of $(A_n)$.
The almost sure limit is denoted $\lambda_j$ and
the $(\lambda_j)$ are the \emph{Lyapunov exponents} of the cocycle.

\section{Exact expressions for Lyapunov exponents}
For $1\leq k \leq d$, we define $e_k$ to be the $k$th coordinate vector, so
that, as previously defined, $c_k(A) := A e_k$ is the $k$th column of $A$. Let $c_k^\perp(A)$ denote the component 
of the $k$th column of $A$ which is orthogonal to the first $k-1$ columns. That is, suppressing the matrix $A$ for brevity,
   $c_1^\perp = c_1$, and 
$$
c_k^\perp := c_k-\sum_{1\leq j<k}\frac{\big\langle c_j^\perp,c_k\big\rangle}{\big\langle c_j^\perp,c_j^\perp\big\rangle}c_j^\perp.
$$

We now prove Theorem \ref{thm:expression}
which we restate here for convenience. 
\begin{thm*}
 Let $(U_n)_{n\in\Z}$ be a sequence of $d\times d$ orthogonal matrices
 and let $(N_n)_{n\in\Z}$ be a sequence of independent $d\times d$ matrices, each with independent
 standard normal coefficients. Let $\epsilon>0$ and let
 $A^\epsilon_j=U_j+\epsilon Z_j$. Then for $1\le k\le d$,
 the $k$th Lyapunov exponent of the cocycle $(A^\epsilon_{\sigma^{n-1}\omega}\cdots
 A^\epsilon_\omega)$ is given by 
 $$
 \lambda_k=\mathbb{E}(\log\big\|c_k^\perp(I+\epsilon N)\big\|).
 $$
\end{thm*}

Fix $\epsilon>0$ and set $A_i:= U_i + \epsilon N_i$ for each $i$.
To find the Lyapunov exponents of this sequence we work with the products $A^{(n)}=A_n\cdots A_1$. 

We now define $\Sigma_n=D(A^{(n)})$ and study the evolution of $\Sigma_n$. 
More precisely we are interested in the stochastic process $(\Sigma_n)_{n\ge 0}$. 
To write $\Sigma_{n+1}$ in terms of $\Sigma_n$, we have 
$\Sigma_{n+1}=D\big(A_{n+1}L(A^{(n)})\Sigma_n R(A^{(n)})\big)$. 
The following lemma shows that this process $(\Sigma_n)$ is Markov
and that the process has the same distribution as the simpler process
$\Sigma'_{n+1}=D((1+\epsilon N_{n+1})\Sigma'_n)$.

\begin{lemma}\label{lem:orthoreduce}($(\Sigma_n)$ is a Markov process)
Let the sequence of matrices $(A_i)$ be given by $U_i+\epsilon N_i$ as above
and let $\Sigma_n=D(A^{(n)})$. Then $(\Sigma_n)$ is a Markov process:
For any measurable set $F$ of diagonal matrices,
\begin{align*}
\mathbb{P}(\Sigma_{n+1}\in F|\Sigma_n,\ldots,\Sigma_1)&=
\mathbb{P}(\Sigma_{n+1}\in F|\Sigma_n)\\
&=\mathbb{P}(D((I+\epsilon N)\Sigma_n)\in F|\Sigma_n).
\end{align*}

That is, the Markov process $(\Sigma_n)$ has the same distribution as the
Markov process $(\Sigma'_n)$ where $\Sigma'_0=I$
and $\Sigma'_{n+1}=D(A'_{n+1}\Sigma_n)$, where $(A'_n)$ is an independent
sequence of matrices, each distributed as $I+\epsilon N$.
\end{lemma}

\begin{proof}
Let $\mathcal{F}_n$ denote the smallest $\sigma$-algebra with respect to which $N_1,\ldots,N_n$ are 
measurable. Let $\mathcal G_n$ be the smallest $\sigma$-algebra with respect to which
$\Sigma_1,\ldots,\Sigma_n$ are measurable (so that $\mathcal G_n$ is a sub-$\sigma$-algebra
of $\mathcal{F}_n$).

As usual, we write $A^{(n)}$ for the product $A_n\ldots A_1$. 
Let $L_n=L(A^{(n)})$, $\Sigma_n=D(A^{(n)})$, $R_n=R(A^{(n)})$. 
Let $F$ be a measurable subset of the range of $D$.
We compute 
\begin{align*}
\mathbb{P}(\Sigma_{n+1}\in F|\mathcal F_n)&=\mathbb{P}\Big(D(A_{n+1}L_n\Sigma_nR_n)\in F|\mathcal F_n\Big)\\
&=\mathbb{P}\Big(D(A_{n+1}L_n\Sigma_n)\in F|\mathcal F_n\Big)\\
&=\mathbb{P}\Big(D\big((U_{n+1}+\epsilon N_{n+1})L_n\Sigma_n\big)\in F|\mathcal F_n\Big)\\
&=\mathbb{P}\Big(D\big(U_{n+1}L_n(I+\epsilon L_n^{-1}U_{n+1}^{-1}N_{n+1}L_n)\Sigma_n\big)\in F|\mathcal F_n\Big)\\
&=\mathbb{P}\Big(D\big((I+\epsilon L_n^{-1}U_{n+1}^{-1}N_{n+1}L_n)\Sigma_n\big)\in F|\mathcal F_n\Big)\\
&=\mathbb{P}\Big(D\big((I+\epsilon N_{n+1})\Sigma_n\big)\in F|\mathcal F_n\Big),
\end{align*}
where the second and fifth lines follow from that facts that $D(A)=D(AU)=D(UA)$ for any matrix $A$ 
and any orthogonal matrix $U$. The sixth line uses the fact that $N_{n+1}$ is independent of $\mathcal F_n$
and Lemma \ref{lem:samedist} so that conditioned on $\mathcal F_n$, $L_n^{-1}U_{n+1}^{-1}N_{n+1}L_n$ has
the same distribution as $N_{n+1}$.
Since $N_{n+1}$ is independent of $\mathcal F_n$, this is equal to 
$\mathbb{P}\Big(D\big((I+\epsilon N_{n+1})\Sigma_n\big)\in F|\Sigma_n\Big )$. 
We have established that
$$
\mathbb{P}(\Sigma_{n+1}\in F|\mathcal F_n)=\mathbb{P}\Big(D\big((I+\epsilon N_{n+1})\Sigma_n\big)\in F|\Sigma_n\Big ).
$$
Taking conditional expectations of both sides with respect to $\mathcal G_n$,
we deduce 
$$
\mathbb{P}(\Sigma_{n+1}\in F|\Sigma_n,\ldots,\Sigma_1)=
\mathbb{P}(\Sigma_{n+1}\in F|\Sigma_n).
$$
\end{proof}

    \begin{proof}[Proof of Theorem \ref{thm:expression}]
        
Fix $1\leq k \leq d$. We use the stochastic process described in Lemma \ref{lem:orthoreduce}:
let $A_n=I+\epsilon N_n$, $\Sigma_0=I$, $\Sigma_{n}=D(A_n\Sigma_{n-1})
=\text{diag}(s_1(A_n\Sigma_{n-1}),\ldots,s_d(A_n\Sigma_{n-1}))$.
As before, we write $A^{(n)}=A_n\ldots A_1$. We note that $\Sigma_n$ is 
\emph{not} equal to $D(A^{(n)})$, but using Lemma \ref{lem:orthoreduce}
the two processes $(\Sigma_n)_{n\ge 0}$ and $\big(D(A^{(n)})\big)_{n\ge 0}$ have
the same distribution.

Let $B_n = A_n\Sigma_{n-1}\begin{pmatrix} e_1\; \dots \; e_k \;0 \;\dots \; 0
\end{pmatrix}$.
Then for all $1\leq j \leq k$,
\begin{align*}
    \big|s_j(\Sigma_{n})-s_j(B_n)\big|
    &=\big|s_j(A_n\Sigma_{n-1})-s_j(B_n)\big|\\
    &\leq \big\|A_n\Sigma_{n-1}-B_n\big\|  \\
    &=
    \norm{
    A_{n}\Sigma_{n-1}\begin{pmatrix} 0\; \dots 0\; e_{k+1} \;\dots \; e_d\end{pmatrix}}
    \\
    &=
    \norm{
    A_{n}\,\text{diag}(0,\ldots 0,s_{k+1}(\Sigma_{n-1}),\ldots ,s_d(\Sigma_{n-1}))
    } \\
    &\leq s_{k+1}(\Sigma_{n-1})\norm{A_{n}}
\end{align*}
Then we have
\begin{align*}
    \left|\frac{s_j(B_n)}{s_j(\Sigma_n)}-1\right|
    &=\left|\frac{s_j(B_n)-s_j(\Sigma_n)}{s_j(\Sigma_n)}\right| \\
    &\leq
    \frac{s_{k+1}(\Sigma_{n-1})}{s_j(\Sigma_n)}\norm{A_{n}}
\end{align*}
By Gol'dsheid and Margulis, \cite[Theorem 5.4]{GoldsheidMargulis}, $\frac{1}{n}\log s_j(A^{(n)})\to \lambda_j$ and 
$\frac{1}{n}\log s_{k+1}(A^{(n)})\to \lambda_{k+1}$ almost surely for some $\lambda_j>\lambda_{k+1}$. 
Since the processes $\big(D(A^{(n)})\big)$ and $(\Sigma_n)$ have a common distribution, it follows that
$\frac{1}{n}\log s_j(\Sigma_n)\to \lambda_j$ and 
$\frac{1}{n}\log s_{k+1}(\Sigma_n)\to \lambda_{k+1}$ almost surely.
So
$$\frac{1}{n}\log\left(\frac{s_{k+1}(\Sigma_{n-1})}{s_j(\Sigma_n)}\right)\to \lambda_{k+1}-\lambda_j<0$$
almost surely as $n\to \infty$. If this occurs, there is some $N\in \N$ such that 
$s_{k+1}(\Sigma_{n-1})/s_j(\Sigma_n)<e^{-n(\lambda_{k+1}-\lambda_j)/2}$ for all $n\geq N$. 
A well-known consequence of the Strong Law of Large Numbers ensures that $C(\omega):=\sup_n \norm{A_n}/n$ is finite a.s., 
so that $\norm{A_n}/n\leq C(\omega)$ for all $n$. For $n\geq N$ we then have
$$
\left|\frac{s_{k+1}(\Sigma_{n-1})}{s_j(\Sigma_n)}\right|\norm{A_{n}}
\leq C(\omega)n e^{-n(\lambda_{k+1}-\lambda_j)/2}\to 0
$$
as $n\to\infty$. Hence 
\begin{equation}\label{eq:asymp}
\frac{s_j(B_n)}{s_j(\Sigma_n)}\to 1 \text{ as $n\to \infty$}.
\end{equation}
For a matrix $A$, let $s_1^k (A)=s_1(A)\cdots s_k(A)$.
Since $B_n$ has $k$ non-zero columns, ${B_n}^{\wedge k}$ has rank one and we have
\begin{align*}
s_1^k(B_n) 
&= \norm{B_ne_1
\wedge B_ne_2
\wedge\dots
\wedge B_ne_k } \\
&= \norm{(A_n\Sigma_{n-1})e_1
\wedge (A_n\Sigma_{n-1})e_2
\wedge\dots
\wedge (A_n\Sigma_{n-1})e_k } \\
&= \norm{(A_{n}e_1)s_1(\Sigma_{n-1})
\wedge (A_{n}e_2)s_2(\Sigma_{n-1}) 
\wedge\dots
\wedge (A_{n}e_k)s_k(\Sigma_{n-1}) } \\
&=s_1^k(\Sigma_{n-1})
\norm{c_1(A_n)
\wedge c_2(A_n)
\wedge\dots
\wedge c_k(A_n)} \\
&=s_1^k(\Sigma_{n-1})
\big\|c_1^\perp(A_n)\big\|\big\|c_2^\perp(A_n)\big\|\cdots \big\|c_k^\perp(A_n)\big\|,
\end{align*}
where $\wedge$ denotes the wedge product. 

For $n\in\N$ and $1\le k\le d$, let $X^k_{n}:=\big\|c_1^\perp(A_n)\big\|
\big\|c_2^\perp(A_n)\big\|\cdots \big\|c_k^\perp(A_n)\big\|$.
Then $X^k_1,X^k_2,\dots$ is a sequence of i.i.d. random variables. 
Since $\Theta(A)\le \|c_i^\perp(A)\|\le \|A\|$, we see, using
Lemma \ref{lem:Thetaest} and Corollary \ref{cor:Margulis-suff}
that $\log\|c_i^\perp(A)\|$ is integrable. We have
\begin{align*}
    s_1^k(\Sigma_n)&=\frac{s_1^k(\Sigma_n)}{s_1^k(B_n)}s_1^k(B_n)\\
    &=\frac{s_1^k(\Sigma_n)}{s_1^k(B_n)}X_n^ks_1^k(\Sigma_{n-1}).
\end{align*}
Using induction, we obtain
$$
s_1^k(\Sigma_n)=\left(\prod_{j=1}^n\frac{s_1^k(\Sigma_j)}{s_1^k(B_j)}\right)X^k_1\ldots X^k_n.
$$
Hence
\begin{align*}
    \tfrac 1n\log s_1^k(\Sigma_n)&=
    \frac 1n\sum_{j=1}^n\log\frac{s_1^k(\Sigma_j)}{s_1^k(B_j)}+\frac 1n\sum_{j=1}^n \log X^k_j.
\end{align*}
By \eqref{eq:asymp}, the first term on the right side converges almost surely
to 0 and by the Strong Law of Large Numbers the second term converges almost surely to $\E\log X^k_1$.
Hence we obtain
$$
\lambda_1+\ldots+\lambda_k=\E\big(\log \|c_1^\perp(I+\epsilon N)\|+\ldots+
\log\|c_k^\perp(I+\epsilon N)\|\big).
$$
Subtracting the $(k-1)$-fold partial sum from the $k$-fold partial sum, we
obtain
$$
\lambda_k=\E\log\|c_k^\perp(I+\epsilon N)\|,
$$
as required.
\end{proof}

	This gives us an explicit description of $\lambda_k$. However it is difficult to compute for large matrices. In the next section we find an approximation for $\lambda_k$ which is easier to compute.
	
	\section{An approximation for $\lambda_j$}
	In this section we focus on the case where $A\sim I_d+\epsilon N$ and introduce the 
   computationally simpler vectors $c_j'(A)$ approximating $c_j^\perp(A)$, defined by $c_1'(A)=c_1(A)$ and
    $$
    c_k'(A)=c_k(A)-\sum_{1\leq j<k}\inprod{c_j(A),c_k(A)}c_j(A)
    $$
    With the same setup as in the previous section, when $|\epsilon\log\epsilon|<(100d)^{-1}$ we have
    \begin{theorem}\label{thm:lognormdiff}
        For any $d\in\N$, if $A_1\sim I_d+\epsilon N$ and $1\leq k\leq d$ then $\mathbb{E}\log\big\|{c_k^\perp}\big\|=\mathbb{E}\log\big\|{c_k'}\big\|+O(\epsilon^{4}|\log\epsilon|^4)$.
    \end{theorem}
 
    We will say that $A=I+\epsilon N$ is \textit{bad} if $|N_{ij}|>|\log\epsilon|$ 
    for some $i,j$. Let $\bad$ denote the event that $A$ is bad. We first control 
    the contribution to
    $\mathbb{E}\log\norm{c_k^\perp}-\mathbb{E}\log\norm{c_k'}$ coming
    from the bad set.
    \begin{lemma}\label{lem:Enormdiffbad}
    Let $\epsilon>0$. Then
    \begin{align*}
    \E\big(\mathbbm 1_{\bad}\big|\log\|c_j^\perp(I+\epsilon N)\|\big|\big)&=O(|\log\epsilon|e^{-(\log\epsilon)^2/2});\text{ and}\\
    \E\big(\mathbbm 1_{\bad}\big|\log\|c_j'(I+\epsilon N)\|\big|\big)&=O(|\log\epsilon|e^{-(\log\epsilon)^2/2}).
    \end{align*}
    \end{lemma}
    
    \begin{proof}
    We write $c_j^\perp$ and $c_j'$ for $c_j^\perp(I+\epsilon N)$ 
    and $c_j'(I+\epsilon N)$ respectively.
    We control the positive parts $\log^+\|c_j'\|$ and $\log^+\|c_j^\perp\|$, and the negative parts
    $\log^-\|c_j'\|$ and $\log^-\|c_j^\perp\|$. For the positive parts, notice that $\|c_j^\perp\|\le \|c_j\|\le \sum_{i,j}|a_{ij}|$
    and $\|c_j'\|\le \Big(1+\sum_{i,j}|a_{ij}|\Big)^3$. The set $\bad$ is a union of $d^2$ parts of the form
    $\bad_{ij}=\{N\colon |N_{ij}|>|\log\epsilon|\}$. Using the bound $\log^+(x)\le x$, this gives 
    \begin{align*}
    \E\big(\mathbbm 1_{\bad}\log^+\|c_j^\perp\|\big)&\le \sum_{i,j}\int_{\bad_{i,j}}\Big(d+\epsilon\sum_{k,l}|x_{kl}|\Big)f_X\big((x_{kl})\big)d(x_{kl})\\
    &\le d^2\int_{\bad_{1,1}}\left(d+\epsilon d^2|x_{11}|\right)f_{N}(x_{11})dx_{11}\\
    &=O(\exp(-(\log\epsilon)^2/2)).
    \end{align*}
    A similar argument holds for $\E\big(\mathbbm 1_{\bad}\log^+\|c_j'\|\big)$.

    To control $\E\big(\mathbbm 1_{\bad}\log^-\|c_j^\perp\|\big)$ and $\E\big(\mathbbm 1_{\bad}\log^-\|c_j'\|\big)$,
    recall $\|c_j^\perp\|$ and $\|c_j'\|$ are bounded below by $\Theta(A)$. 
    By standard estimates on the tail of the normal distribution, 
    $\PP(\bad)=O(e^{-(\log\epsilon)^2/2}/|\log\epsilon|)$.
    We see from Lemma \ref{lem:Thetaest}, $\E\log^-\Theta(I+\epsilon N)\mathbf 1_{\bad}=
    O(|\log\epsilon|e^{-(\log\epsilon)^2/2})$,
    which gives the required estimates.

    \end{proof}

    We now give pointwise estimates for $\big|\log\|{c_k^\perp}\|-\log\norm{c_k'}\big|$ when $A$ is not bad. That is, when $A=I+\epsilon N_{ij}$
    where $|N_{ij}|\leq |\log\epsilon|$ for all $i,j$.

    \begin{lemma}\label{lem:qdiff}
    There exist $\epsilon_0>0$ and $C>0$ depending only on $d$
    such that for all matrices $A$ of the form $A=I+\epsilon X$ where $|X_{ij}|\le |\log\epsilon|$ for each $i,j$, then for each $k$,
    $$
    \Big|\log\big\|c_k^\perp(A)\big\|-\log\big\|c_k'(A)\|\Big|\le C(\epsilon|\log\epsilon|)^4\text{ for all $\epsilon<\epsilon_0$.}
    $$
    \end{lemma}

    As usual, we write $c_j$, $c_j^\perp$ and $c_j'$ in place of $c_j(A)$, $c_j^\perp(A)$ and 
    $c_j'(A)$ for brevity. 
    We define $\alpha_i^j :=\frac{\inprod{c_i^\perp,c_j}}{\norm{c_i^\perp}^2}$ so that
    $c_j^\perp = c_j-\sum_{i<j}\alpha_i^j c_i^\perp$. Throughout the proof, we let 
    $\eta=|\log\epsilon|$. We let $\epsilon_0$ be sufficiently small that 
    $\epsilon\eta<1/(100d)$ for $\epsilon<\epsilon_0$.
    The proof makes use of a number of claims.

     \bigskip
\begin{claim} Let $A=I+\epsilon X$ where $|X_{ij}|\le\eta$ for all $i,j$.\label{claim:one}
For all $1\leq n\leq d$, the following hold:
    \begin{enumerate}[(i)]
        \item $|\norm{c_n}^2-1|\leq 2\epsilon\eta+d\eta^2\epsilon^2\leq 3\epsilon\eta$;\label{it:claim1(i)}
        \item $|\norm{c_n^\perp}^2-1|\leq 3\epsilon \eta$;\label{it:claim1(ii)}
        \item $|\alpha_i^k|\leq 6\epsilon\eta$ for all $i\le n$ and $k>i$;\label{it:claim1(iii)}
        \item $|\inprod{c_n^\perp,c_k}|\leq 3\epsilon \eta$ for all $k>n$.\label{it:claim1(iv)} 
    \end{enumerate}
\end{claim}

    \bigskip\noindent
\begin{proof}
    Since $|X_{ij}|\le\eta$ for all $i,j$, 
     for any $1\leq n\leq d$ and $i<j$ we have
     \begin{align*}
     |\norm{c_n}^2-1|&\leq 2\epsilon\eta+d\epsilon^2\eta^2\quad \text{and}\\
         |\inprod{c_i,c_j}|&\leq 2\epsilon \eta+d\epsilon^2\eta^2.
    \end{align*}
     This shows \eqref{it:claim1(i)} for all $n$, as well as \eqref{it:claim1(ii)},
     \eqref{it:claim1(iii)} and \eqref{it:claim1(iv)} for $n=1$. 
     
     Now suppose for some $2\le j\le d$, \eqref{it:claim1(ii)}--\eqref{it:claim1(iv)} each hold for all $n\leq j-1$. 
     Then for all $k>j$ we have
     \begin{align*}
         \big\langle c_j^\perp, c_k\big\rangle=\Big\langle c_j-\sum_{i<j}\alpha_i^j c_i^\perp, c_k\Big\rangle
         =\big\langle c_j, c_k\big\rangle-\sum_{i<j}\alpha_i^j\big\langle c_i^\perp, c_k\big\rangle
     \end{align*}
     This implies
     \begin{align*}
         \big|\big\langle c_j^\perp,c_k\big\rangle\big|
         &\leq |\big\langle c_j,c_k\big\rangle|+\sum_{i<j}|\alpha_i^j|
         \big|\big\langle c_i^\perp,c_k\big\rangle\big| \\
         &\leq (2\epsilon \eta+d\epsilon^2\eta^2)+ d \cdot (6\epsilon \eta )(3\epsilon \eta)\\
         &\leq 3\epsilon \eta,
     \end{align*}
     where we used \eqref{it:claim1(i)} and
     the induction hypotheses in the second line and the condition
     on $\epsilon_0$ in the third line. This establishes \eqref{it:claim1(iv)} for $n=j$.

    Since $c_1^\perp,\dots,c_j^\perp$ are mutually perpendicular, it follows that
     \begin{align*}
        \big\|c_j\big\|^2 = \big\|c_j^\perp\big\|^2+\sum_{i<j}(\alpha_i^j)^2 \big\|c_i^\perp\big\|^2
     \end{align*}
     Thus we have
    \begin{align*}
         \Big|\big\|c_j^\perp\big\|^2-1\Big|
         &=\bigg|\big\|c_j\|^2-1+\sum_{i<j}(\alpha_i^j)^2 \big\|c_i^\perp\big\|^2 \bigg|
          \\
          &\leq 
          \Big|\big\|c_j\big\|^2-1\Big|+\sum_{i<j}(\alpha_i^j)^2 \big\|c_i^\perp\big\|^2 
          \\
         &\leq(2\epsilon\eta+d\epsilon^2\eta^2)+ d(6\epsilon\eta)^2(1+3\epsilon\eta)
          \\
         &\leq 3\epsilon\eta,
     \end{align*}
    establishing \eqref{it:claim1(ii)} for $n=j$. 

    We show that \eqref{it:claim1(iii)} holds for $n=j$. Since by the induction hypothesis,
    $|\alpha_i^k|\le 6\epsilon\eta$ for all $i<j$ and $k>i$, 
    if suffices to show that $|\alpha_j^k|\le 6\epsilon\eta$ for all $k>j$.
    For any $k>j$, using \eqref{it:claim1(iv)}, we have
     \begin{align*}
         \big|\alpha_j^k\big|=\frac{|\big\langle c_j^\perp,c_k\big\rangle|}{\big\|c_j^\perp\big\|^2}\leq \frac{3\epsilon\eta}{1/2}=6\epsilon\eta
     \end{align*}
     which shows that (iii) holds for $n=j$. 
\end{proof}
    
\begin{claim} For each $1\leq n\leq d$, $c_n^\perp = c_n + \sum_{j<n} \beta_{j}^n c_j$ where $|\beta_j^n|<7\epsilon \eta$.
\label{claim:two}
\end{claim}

\begin{proof}
    We use induction on $j$. 
    The base case is $c_1^\perp=c_1$. Suppose the claim holds for all $n<j\leq d$. Then 
     \begin{align*}
         c_j^\perp 
         &= c_j - \sum_{i<j}\alpha_i^j c_i^\perp\\
         &=c_j - \sum_{i<j}\alpha_i^j \Big(c_i+\sum_{\ell<i}\beta_\ell^i c_\ell\Big)\\
         &=c_j - \sum_{\ell<j}\alpha_\ell^j c_\ell-\sum_{i<j}\alpha_i^j\sum_{\ell<i}\beta_\ell^i c_\ell \\
         &= c_j - \sum_{\ell<j}\alpha_\ell^j c_\ell -\sum_{\ell<j-1}\Big(\sum_{i=\ell+1}^{j-1}\alpha_{i}^j \beta_\ell^i\Big) c_\ell
     \end{align*}
     For any $\ell<j$, the coefficient of $c_\ell$ in the above expression is bounded by 
     $$
    |\alpha_\ell^j|+\sum_{i=\ell+1}^{j-1}|\alpha_i^j \beta_\ell^i|
    \leq 6\epsilon \eta + d(6\epsilon \eta)(7\epsilon\eta)\leq 7\epsilon\eta
$$
\end{proof}

\begin{claim}\label{claim:three}
    For all $1\leq j\leq d$, $c_j' = c_j^\perp+\sum_{n<j}\gamma_n c_n$ where $\gamma_n=O(\epsilon^2\eta^2)$,
    where the implicit constant depends only on $d$
\end{claim}

\begin{proof}
   For any such $j$ we have 
    $$
    c_j'-c_j^\perp = \sum_{i<j} \bigg(
    \frac{\inprod{c_i^\perp,c_j}}{\inprod{c_i^\perp,c_i^\perp}}c_i^\perp-\inprod{c_i,c_j}c_i\bigg).
    $$

We identify the coefficient of $c_\ell$ when $c_j'-c_j^\perp$ is expanded in the basis $(c_k)$.
That coefficient may be seen to be 
\begin{align*}
&\frac{\inprod{c_\ell^\perp,c_j}}{\inprod{c_\ell^\perp,c_\ell^\perp}}-\inprod{c_\ell,c_j}
+\sum_{\ell<i< j}\frac{\inprod{c_i^\perp,c_j}}{\inprod{c_i^\perp,c_i^\perp}}\beta^i_\ell\\
=&\frac{\inprod{c_\ell^\perp,c_j}-\inprod{c_\ell,c_j}}
{\inprod{c_\ell^\perp,c_\ell^\perp}}
+\frac{\inprod{c_\ell,c_j}
\big(1-\inprod{c_\ell^\perp,c_\ell^\perp}\big)}{\inprod{c_\ell^\perp,c_\ell^\perp}}
+O(\epsilon^2\eta^2),
\end{align*}
where we added and subtracted $\inprod{c_\ell,c_j}/\inprod{c_\ell^\perp,c_\ell^\perp}$;
and the estimate for the third term follows from Claims \ref{claim:one} and
\ref{claim:two}.

Since $\inprod{c_\ell^\perp,c_j}-\inprod{c_\ell,c_j}=
-\inprod{\sum_{i<\ell}\beta^\ell_ic_i,c_j}$,
the estimates in Claims \ref{claim:one} and \ref{claim:two} show the first term is  $O(\epsilon^2\eta^2)$.
Finally since $\inprod{c_\ell,c_j}=O(\epsilon\eta)$ and $1-\inprod{c_\ell^\perp,c_\ell^\perp}$ is $O(\epsilon\eta)$ by Claim \ref{claim:one}, the 
middle term is also $O(\epsilon^2\eta^2)$.

\end{proof}
    
\begin{proof}[Proof of Lemma \ref{lem:qdiff}]
By orthogonality,
    \begin{align*}
        \big\|c_j'\big\|^2 = 
        \bigg\|c_j^\perp+\sum_{n<j}\gamma_n c_n\bigg\|^2
        =
        \big\|c_j^\perp\|^2+\bigg\|\sum_{n<j}\gamma_n c_n\bigg\|^2,
    \end{align*}
where $\gamma_n$ is as in Claim \ref{claim:three}.
Since $\gamma_n=O(\epsilon^2\eta^2)$, we obtain
$\big\|c_j'\big\|^2 = \big\|c_j^\perp\|^2+ O(\epsilon^4\eta^4)$.
Since $\big\|c_j^\perp\|^2$ is in the range $(\frac 12,\frac 32)$, 
it follows that $\big|\log\big\|c_j'\big\|-\log\big\|c_j^\perp\big\|\big|=O(\epsilon^4\eta^4)$ 
as required.
    
\end{proof}

    \begin{proof}[Proof of Theorem \ref{thm:lognormdiff}]
    Lemma \ref{lem:Enormdiffbad} shows that 
    \begin{align*}
    &\big|\mathbb{E}\big(\log\|c_k^\perp\|-\log\|c_k'\|)\mathbf 1_\bad\big)\big|\\
    \le{} &\mathbb{E}\big(\log\|{c_k^\perp}\|\mathbf 1_\bad\big)+\mathbb{E}\big(\log\|{c_k'}\|\mathbf 1_\bad\big)
    \\   
    ={}&O(|\log\epsilon|e^{-(\log\epsilon)^2}).
    \end{align*}
    and Lemma \ref{lem:qdiff} shows that
    $\big|\log\big\|{c_k^\perp}\big\|-\log\big\|{c_k'}\big\|\big|\mathbf 1_{\bad^c}=O(\epsilon^4|\log\epsilon|^4)$. 
    Taking the expectation of this and combining the estimates gives the theorem.
    \end{proof}

	\section{Computing $\mathbb{E}\log\|c_k'\|$}
    Finally, we find the dominant term in the asymptotic expansion for $\mathbb{E}\log\|c_j'\|$ in the same setup as the previous section. This is
    Theorem \ref{thm:approx} which we restate here for convenience. 
    \begin{thm*}
    Consider an orthogonal-plus-Gaussian cocycle as in Theorem 
    \ref{thm:expression}. Then the Lyapunov exponents satisfy
    $$
    \lambda_k(\epsilon)=(d-2k)\tfrac{\epsilon^2}2+O(\epsilon^4|\log\epsilon|^4)
    \text{ as $\epsilon\to 0$.}
    $$
    \end{thm*}

	As in the previous sections, let $A = I_d + \epsilon N$ where $N$ is a standard
    Gaussian matrix random variable. 

    \begin{proof}
     Let $\eta=|\log\epsilon|$ and let $\bad$ be defined as above. 
     We assume $\epsilon$ is sufficiently small that $\|c_j'(I+\epsilon N)\|^2\in (\frac 12,\frac 32)$
     for all $N\in\bad^c$. Expanding, we have that 
    \begin{align*}
        &\quad\|c_j'\|^2=\bigg\langle
        c_j-\sum_{i<j}\langle c_i,c_j\rangle c_i\,,\,c_j-\sum_{k<j}\langle c_k,c_j\rangle c_k
        \bigg\rangle\\
        &=\|c_j\|^2 - 2\sum_{i<j}\langle c_i,c_j\rangle^2
        +\sum_{i,k<j}\langle c_i,c_j\rangle\langle c_k,c_j\rangle
        \langle c_i,c_k\rangle\\
        &=\|c_j\|^2 - 2\sum_{i<j}\langle c_i,c_j\rangle^2
        +\sum_{i<j}\langle c_i,c_j\rangle^2\|c_i\|^2
        +2\sum_{i<k<j}\langle c_i,c_j\rangle\langle c_k,c_j\rangle
        \langle c_i,c_k\rangle\\
        &=\|c_j\|^2 - \sum_{i<j}\langle c_i,c_j\rangle^2(2-\|c_i\|^2)
        +2\sum_{i<k<j}\langle c_i,c_j\rangle\langle c_k,c_j\rangle
        \langle c_i,c_k\rangle,
        \end{align*}
    where to obtain the third line from the second, we separated the case $i=k$ from the
    case $i\ne k$.
        
    We take a finite Taylor expansion, valid for $t\in(-1,1)$:
     $\log(1+t) = t - \frac{t^2}{2}+\frac{t^3}3 - R(t)$ where $R(t)=\frac{1}{4} (1+\xi)^{-4} t^4$
    for some $\xi$ with $|\xi|\leq |t|$. 
    Let $X_j$ be the random variable $\|c_j'(I+\epsilon N)\|^2-1$. Notice from the above
    that $X_j$ is a polynomial
    of degree 6 (whose coefficients don't depend on $\epsilon$)
    in the entries of $\epsilon N$. If $N=0$, then $c_j'(I+\epsilon N)=e_j$
    so that the constant term in the polynomial $X_j$ is 0. Notice also that
    by Claim \ref{claim:one}, on $\bad^c$, 
    all terms other than the first term in the expression for $\|c_j'\|^2$
    are $O(\epsilon^2|\log\epsilon|^2)$,
    while a calculation shows that $\|c_j\|^2=1+O(\epsilon|\log\epsilon|)$.
    Hence $X_j\mathbf 1_{\bad^c}=O(\epsilon|\log\epsilon|)$.
    Let $Y_j=X_j-\frac 12X_j^2+\frac 13X_j^3$, so that $Y_j$ is another polynomial
    in the entries of $\epsilon N$ with no constant term.
    Combining the above, on $\bad^c$
    $$
    \log(\|c_j'(I+\epsilon N)\|^2)=\log(1+X_j)=Y_j+O(\epsilon^4|\log\epsilon|^4).
    $$
 
    Then we have 
    \begin{equation}\label{eq:Eapprox}
    \begin{split}
        &\E\log(\|c_j'(I+\epsilon N)\|^2)\\
        ={}&\E\log(\|c_j'(I+\epsilon N)\|^2\mathbf 1_{\bad^c})
        +\E\log(\|c_j'(I+\epsilon N)\|^2\mathbf 1_\bad)\\
        ={}&\E(Y_j\mathbf 1_{\bad^c})+O(\epsilon^4|\log\epsilon|^4)+\E\log(\|c_j'(I+\epsilon N)\|^2\mathbf 1_\bad)\\
        ={}&\E Y_j-\E (Y_j\mathbf 1_\bad)+\E\log(\|c_j'(I+\epsilon N)\|^2\mathbf 1_\bad)
        +O(\epsilon^4|\log\epsilon|^4).   
    \end{split}
    \end{equation}
Since $Y_j$ is a fixed polynomial function of the entries of $\epsilon N$, 
and all monomials that are products of entries $N$ have finite expectation,
we see that $\E Y_j$ agrees up to order $\epsilon^4$ with the expectation of its
terms of degree 3 or lower. Also, since the entries of $N$ are independent and each has 
a symmetric distribution, the constant term of $Y_j$ being 0, 
the only terms that give a non-zero contribution to $\E Y_j$ are the
terms of the forms $N_{ab}^2$. Since the lowest order terms in $Y_j$ are
polynomials of degree 1, and $Y_j=X_j-\frac 12X_j^2+\frac 13X_j^3$,
the terms of the form $N_{ab}^2$ in $Y_j$ are those appearing in $X_j$ and $\frac 12X_j^2$.

We established above
$$
X_j=\|c_j\|^2-1 - \sum_{i<j}\langle c_i,c_j\rangle^2(2-\|c_i\|^2)
        +2\sum_{i<k<j}\langle c_i,c_j\rangle\langle c_k,c_j\rangle
        \langle c_i,c_k\rangle.
$$
We see that $\|c_j\|^2-1=2\epsilon N_{jj}+\epsilon^2\sum_i N_{ij}^2$
and $\langle c_i,c_j\rangle=\epsilon(N_{ij}+N_{ji})+\epsilon^2\sum_k N_{ki}N_{kj}$.

Substituting these in the expression for $X_j$, we see 
\begin{align*}
    \E X_j&=d\epsilon^2-\epsilon^2\sum_{i<j}\E (N_{ij}+N_{ji})^2+O(\epsilon^4)\\
    &=(d-2j+2)\epsilon^2+O(\epsilon^4).
\end{align*}
We also see $\E X_j^2=4\epsilon^2\E N_{jj}^2+O(\epsilon^4)$.
Combining these gives 
$$
\E Y_j=\E(X_j-\tfrac 12X_j^2)+O(\epsilon^4)=(d-2j)\epsilon^2+O(\epsilon^4).
$$
Therefore by \eqref{eq:Eapprox}, to finish the argument, it suffices to show
    $\E (Y_j\mathbf 1_\bad)=O(\epsilon^4|\log\epsilon|^4)$ 
    and $\E \log(\|c_j'(I+\epsilon N)\|^2\mathbf 1_{\bad})=O(\epsilon^4|\log\epsilon|^4)$.

Since $ \|c_j'(A)\|\ge \Theta(A)$, Lemma \ref{lem:Thetaest} shows 
$$
\E(\log^-\|c_j'(I+\epsilon N)\|\mathbf 1_\bad)
=O(|\log\epsilon|^2e^{-(\log\epsilon)^2/2}).
$$ 
Since $\|c_j'(A)\|\le 2(\sum_{k,l}|A_{kl}|)^3$,
we see $\E(\log^+\|c_j'(I+\epsilon N)\|^2\mathbf 1_\bad)=O(\PP(\bad))=O(e^{-|\log\epsilon|^2/2}/
|\log\epsilon|)$.

Finally, for any of the (finitely many) monomial terms $M$ appearing in $Y_j$, we can check
$\E M\mathbf 1_\bad=O(\PP(\bad))=O(e^{-|\log\epsilon|^2/2}/
|\log\epsilon|)$. This completes the proof.
    


\end{proof}

 \bibliographystyle{abbrv}

\begin{thebibliography}{10}

\bibitem{BaxendaleHarris}
P.~Baxendale and T.~E. Harris.
\newblock Isotropic stochastic flows.
\newblock {\em Ann. Probab.}, 14:1155--1179, 1986.

\bibitem{Bhatia}
R.~Bhatia.
\newblock {\em Matrix analysis}.
\newblock Springer-Verlag, New York, 1997.

\bibitem{Dynkin}
E.~Dynkin.
\newblock Non-negative eigenfunctions of the {L}aplace-{B}eltrami operator and
  {B}rownian motion in certain symmetric spaces.
\newblock {\em Dokl. Akad. Nauk SSSR}, pages 288--291, 1961.

\bibitem{GoldsheidMargulis}
I.~Y. Gol'dsheid and G.~A. Margulis.
\newblock Lyapunov exponents of a product of random matrices.
\newblock {\em Russian Math. Surveys}, 44:11--71, 1989.

\bibitem{horn_johnson_1991}
R.~A. Horn and C.~R. Johnson.
\newblock {\em Topics in Matrix Analysis}.
\newblock Cambridge University Press, 1991.

\bibitem{LeJan}
Y.~Le~Jan.
\newblock On isotropic {B}rownian motions.
\newblock {\em Z. Wahrsch. Verw. Gebiete}, 70:609--620, 1985.

\bibitem{Newman}
C.~M. Newman.
\newblock The distribution of {L}yapunov exponents: exact results for random
  matrices.
\newblock {\em Comm. Math. Phys.}, 103:121--126, 1986.

\bibitem{NorrisRogersWilliams}
J.~R. Norris, L.~C.~G. Rogers, and D.~Williams.
\newblock Brownian motions of ellipsoids.
\newblock {\em Trans. Amer. Math. Soc.}, 294:757--765, 1986.

\bibitem{Oseledets}
V.~I. Oseledec.
\newblock A multiplicative ergodic theorem. {C}haracteristic {L}japunov,
  exponents of dynamical systems.
\newblock {\em Trudy Moskov. Mat. Ob\v{s}\v{c}.}, pages 179--210, 1968.

\bibitem{Raghunathan}
M.~S. Raghunathan.
\newblock A proof of {O}seledec’s multiplicative ergodic theorem.
\newblock {\em Israel J. Math.}, 32:356--362, 1979.

\end{thebibliography}

\end{document}